\documentclass[12pt,a4paper]{amsart}
\usepackage{amsfonts}
\usepackage{amssymb}
\usepackage{amsmath}
\usepackage{amsthm}
\usepackage{cite}
\usepackage{graphicx}
\usepackage{amscd}
\usepackage{color}
\newtheorem{theorem}{Theorem}
\theoremstyle{plain}

\newtheorem{corollary}[theorem]{Corollary}

\newtheorem{lemma}[theorem]{Lemma}

\numberwithin{equation}{section}

\newcommand{\R}{\mathbb{R}}

\newcommand{\N}{\mathbb{N}}

\newcommand{\cK}{{\mathcal K}}

\newcommand{\cM}{{\mathcal M}}

\renewcommand{\phi}{\varphi}

\DeclareMathOperator{\Id}{Id}
\DeclareMathOperator{\degree}{deg}

\newcommand{\dist}{\text{\rm dist}}

\begin{document}

\title[Existence for nonlinear higher-order elliptic problems]{Existence of solutions to nonlinear, subcritical higher-order elliptic Dirichlet problems} 

\author{Wolfgang Reichel}
\address{W. Reichel \hfill\break 
Institut f\"ur Analysis, Universit\"at Karlsruhe, \hfill\break
D-76128 Karlsruhe, Germany}
\email{wolfgang.reichel@math.uni-karlsruhe.de}

\author{Tobias Weth}
\address{T. Weth \hfill\break 
Institut f\"ur Mathematik, Johann Wolfgang Goethe-Universit\"at Frankfurt\hfill\break
D-60054 Frankfurt, Germany}
\email{weth@math.uni-frankfurt.de}
\date{\today}


\begin{abstract} 
We consider the $2m$-th order elliptic boundary value problem 
$Lu=f(x,u)$ on a bounded smooth domain $\Omega\subset\R^N$ with Dirichlet boundary conditions on 
$\partial\Omega$. The operator $L$ is a uniformly elliptic linear operator of order 
$2m$ whose principle part is of the form $\big(-\sum_{i,j=1}^N a_{ij}(x) 
\frac{\partial^2}{\partial x_i\partial x_j}\big)^m$. 
We assume that $f$ is superlinear at the origin and satisfies $\lim \limits_{s\to\infty}\frac{f(x,s)}{s^q}=h(x)$, 
$\lim \limits_{s\to-\infty}\frac{f(x,s)}{|s|^q}=k(x)$, where $h,k\in C(\overline{\Omega})$ are positive functions and $q>1$ is subcritical. By combining degree theory with new and recently established a priori estimates, we prove the existence 
of a nontrivial solution. 

\end{abstract}

\maketitle



\section{Introduction}

Let $\Omega\subset\R^N$ be a bounded smooth domain. On $\Omega$ we consider the uniformly elliptic operator
\begin{equation}
L = \Bigl(- \sum_{i,j=1}^N a_{ij}(x) 
\frac{\partial^2}{\partial x_i\partial x_j}\Bigr)^m 
+\sum_{0\leq |\alpha|\leq 2m-1} b_\alpha(x) D^\alpha 
\label{operator}
\end{equation}
with coefficients $b_\alpha\in C^\alpha(\overline{\Omega})$ and 
$a_{ij}\in C^{2m-2,\alpha}(\overline{\Omega})$ such that there exists a 
constant $\lambda>0$ with $\lambda^{-1} |\xi|^2 \leq \sum_{i,j=1}^N a_{ij}(x)\xi_i\xi_j \leq 
\lambda |\xi|^2$ for all $\xi\in \R^N$, $x\in \Omega$. We are interested in nontrivial solutions of the semilinear boundary value problem 
\begin{equation}
L u = f(x,u) \mbox{ in } \Omega, \quad u=\frac{\partial}{\partial\nu}u=
\ldots=\left(\frac{\partial}{\partial\nu}\right)^{m-1} u =0 \mbox{ on } 
\partial\Omega,
\label{basic}
\end{equation}
where $\nu$ is the unit exterior normal on $\partial\Omega$ and $f$ is a nonlinearity which is to be specified later. The main difficulties in proving existence results for this problem are the following:
\begin{itemize}
\item[1.] \eqref{basic} has no variational structure (in general), so critical point theorems do not apply;
\item [2.] The operator $L$ does not satisfy the maximum principle (in general) unless $m=1$. In the second order case, the maximum principle is a basic 
requirement to translate \eqref{basic} into a fixed point problem for an order preserving operator, which in turn makes it possible to use topological degree (or fixed point) theory in cones or invariant order intervals given by a pair of sub- and supersolutions. 
\item [3.] In the case $m>1$, a priori bounds for (certain classes of) solutions are harder to obtain than in the second order case, which makes it 
difficult to find solutions to \eqref{basic} via global bifurcation theory.   
\end{itemize}
In a recent paper, we have proved a priori bounds for solutions of \eqref{basic} in the case of superlinear nonlinearities $f(x,u)$ with subcritical growth 
satisfying an asymptotic condition. More precisely, we assumed:
\begin{itemize}
\item[(H1)] $f:\Omega\times\R\to\R$ is uniformly continuous in bounded subsets of $\Omega\times\R$ and there exists $q>1$ if $N\leq 2m$ and $1<q<\frac{N+2m}{N-2m}$ if $N>2m$ and two positive, continuous functions $k,h: \overline{\Omega}\to (0,\infty)$ such that 
$$
\lim_{s\to +\infty} \frac{f(x,s)}{s^q} = h(x), \quad 
\lim_{s\to-\infty} \frac{f(x,s)}{|s|^q}= k(x)\qquad \text{uniformly with respect to $x\in \overline{\Omega}$.}
$$
\end{itemize}

\begin{theorem}[Reichel, Weth \cite{ReichelWeth}]
\label{reichel-weth-alt}
If $f:\Omega\times\R\to\R$ satisfies (H1) then there exists a constant $C>0$ depending only on the data 
$a_{ij}, b_\alpha, \Omega, N, q, h, k$ such that $\|u\|_\infty \leq C$ for 
every solution $u\in C^{2m,\alpha}(\overline{\Omega})$ of \eqref{basic}. 
\end{theorem}

This result can be seen as a first step towards existence results via degree theory. In order to state the main theorem of the present paper, we introduce additional assumptions on $f$. 

\begin{itemize}
\item[(H2)] For all $x\in \Omega$ the function $f(x,s)$ is continuously differentiable with respect to $s$ and $f(x,s)$, $\partial_s f(x,s)$ are $\alpha$-H\"older continuous in $x$ uniformly for $x\in\Omega$ and $s$ in bounded intervals. Moreover, $f(x,0)=\partial_s f(x,0)=0$. 
\item[(H3)] The operator $L$ has a bounded inverse $L^{-1}$ which maps $C^\alpha(\overline{\Omega})\to C^{2m,\alpha}(\overline{\Omega})$ with Dirichlet boundary conditions of order up to $m-1$.
\end{itemize}

\begin{theorem}
Suppose $\Omega\subset\R^N$ is a bounded domain with $\partial\Omega\in 
C^{2m,\alpha}$. Let $m\in \N$ and assume $f:\Omega\times\R\to\R$ satisfies (H1), (H2), (H3). Then \eqref{basic} has a nontrivial solution $u\in C^{2m,\alpha}(\overline{\Omega})$.
\label{ex}
\end{theorem}

We note that in many examples condition (H3) can be verified with the help of the Lax-Milgram Theorem and elliptic regularity, see Agmon, Douglis, Nirenberg \cite{ADN}. In particular, if $L$ is as in \eqref{operator}, then $L+\gamma$ satisfies (H3) if $\gamma>0$ is sufficiently large and if additionally $b_\alpha \in C^{|\alpha|-m}(\overline{\Omega})$ for $m<|\alpha|<2m$. This is true since the smoothness of the coefficients allows to write $L$ in divergence form and hence the quadratic form associated with $L+\gamma$ is coercive due to Garding's inequality, cf. Renardy-Rogers \cite{RR}, if $\gamma>0$ is sufficiently large.

As an intermediate step in the proof of Theorem \ref{ex}, we need to complement Theorem \ref{reichel-weth-alt} with the following a priori estimate for a parameter, which might be of independent interest. 

\begin{theorem}
Suppose $\Omega\subset\R^N$ is a bounded domain with $\partial\Omega\in 
C^{2m,\alpha}$. Let $m\in \N$ and assume $f:\Omega\times\R\to\R$ satisfies (H1). Then there exists a value $\Lambda=\Lambda(\Omega,L,f)$ such that for $\lambda\geq\Lambda$ the problem 
\begin{equation}
L u = f(x,u)+\lambda \mbox{ in } \Omega, \quad u=\frac{\partial}{\partial\nu}u=
\ldots=\left(\frac{\partial}{\partial\nu}\right)^{m-1} u =0 \mbox{ on } 
\partial\Omega.
\label{basic_lambda}
\end{equation}
has no solution $u\in C^{2m,\alpha}(\overline{\Omega})$.
\label{non_ex}
\end{theorem}

Due to the lack of the maximum principle for higher order equations, we have no sign information on the solution provided by Theorem \ref{ex}. By the same reason, it is important that Theorems \ref{reichel-weth-alt} and \ref{non_ex} hold with no restriction on the sign of the solutions. We also point out that we make no assumption concerning the shape of the domain.

We recall that the proof of Theorems \ref{reichel-weth-alt} is carried out by a rescaling method in the spirit of the seminal work of Gidas and 
Spruck \cite{GS1} (but without a priori information on the sign of the solutions) and by investigating the corresponding limit problems. In particular, the following Liouville type theorems are used.

\begin{theorem}[Wei, Xu \cite{WeiXu}] Let $m \in \N$ and assume that $q>1$ if $N\leq 2m$ and 
$1<q<\frac{N+2m}{N-2m}$ if $N>2m$. If $u$ is a classical non-negative solution of 
\begin{equation}
\label{eq-rn}
(-\Delta)^m u = u^q \mbox{ in } \R^N,
\end{equation}
then $u\equiv 0$.
\label{rn}
\end{theorem}

\begin{theorem} 
Let $m \in \N$ and assume that $q>1$ if $N\leq 2m$ and $1<q\leq \frac{N+2m}{N-2m}$ 
if $N>2m$. If $u$ is a classical non-negative 
solution of 
\begin{equation}
  \label{eq:liouville-half-space}
(-\Delta)^m u = u^q \mbox{ in } \R^N_+, \quad u=
\frac{\partial}{\partial x_1}u=\ldots=
\frac{\partial^{m-1}}{\partial x_1^{m-1}} u =0 \mbox{ on }\partial \R^N_+
\end{equation}
then $u\equiv 0$.
\label{rn_plus}
\end{theorem} 
Here and in the following, we set $\R^N_+:= \{x \in \R^N\::\: x_1>0\}$. Theorem \ref{rn_plus} is a slight generalization of Theorem~4 in our recent paper \cite{ReichelWeth}. More precisely, it is assumed in \cite{ReichelWeth} that $u$ is bounded, but an easy argument based on the doubling lemma of Pol\'{a}\v{c}ik, Quittner and Souplet \cite{PQS} shows that this additional assumption can be removed. See Section~\ref{no-boundedness} below for details.

Theorems~\ref{rn} and~\ref{rn_plus} will also be used in the proof of Theorem~\ref{non_ex}. However, the rescaling argument is somewhat more 
involved since both $\lambda$ and the $L^\infty$-norm of the solutions need to be controlled. Here various cases have to be distinguished, and 
additional limit problems arise.

Finally we comment on some previous work related to Theorem~\ref{ex}. If $L=(-\Delta)^m$ is the polyharmonic operator, then \eqref{basic} has a variational structure. In this case existence and multiplicity results for 
solutions of \eqref{basic} have been obtained under additional assumptions on $f$ via critical point theory and related techniques, see e.g. \cite{weth,gazzola-grunau-squassina,grunau,EFJ} and the references therein. The  
approach via a priori estimates and degree theory was taken by Soranzo \cite{soranzo} and Oswald \cite{oswald}, but only in the special case where $\Omega$ is a ball. More precisely, in \cite{soranzo,oswald} the authors first prove a priori estimates for {\em radial positive solutions} before proving existence results within this class of functions. An existence result for more general operators $L$ was obtained in \cite{grunau-sweers-97} for a different class of nonlinearities which gives rise to coercive nonlinear operators. See also the references in \cite{grunau-sweers-97} for earlier results in this direction. 

The paper is organised as follows. Section~\ref{sec:non_ex} is devoted to the proof of Theorem~\ref{non_ex}, while Theorem \ref{ex} is proved in Section~\ref{sec:ex}. Finally, in Section~\ref{no-boundedness} we show how to remove the boundedness assumption which was present in the original formulation of Theorem \ref{rn_plus}.

\section{Nonexistence for the parameter dependent problem} \label{sec:non_ex}

The proof of Theorem~\ref{non_ex} uses standard $L^p$-$W^{2m,p}$ estimates for linear problems
\begin{align}
L u &= g(x) \mbox{ in } \Omega, \label{linear_eq}\\
u &=\frac{\partial}{\partial\nu}u=
\ldots=\left(\frac{\partial}{\partial\nu}\right)^{m-1} u =0 \mbox{ on } 
\partial\Omega. \label{linear_bc}
\end{align}
Recall the following basic estimate of Agmon, Douglis, Nirenberg \cite{ADN}.

\begin{theorem}[Agmon, Douglis, Nirenberg] Let 
$\Omega\subset\R^N$ be a bounded domain with $\partial\Omega\in C^{2m}, 
m\in \N$. Let $a_{ij}\in C^{2m-2}(\overline{\Omega})$, 
$b_\alpha \in L^\infty(\Omega)$, $g\in L^p(\Omega)$ for some 
$p\in (1,\infty)$. Suppose $u\in W^{2m,p}(\Omega)\cap W_0^{m,p}(\Omega)$ 
satisfies \eqref{linear_eq}. Then there exists a constant 
$C>0$ depending only on $\|a_{ij}\|_{C^{2m-2}}, \|b_\alpha\|_\infty, \lambda, 
\Omega, N, p, m$ and the modulus of continuity of $a_{ij}$ such that  
$$
\|u\|_{W^{2m,p}(\Omega)} \leq C(\|g\|_{L^p(\Omega)}+ \|u\|_{L^p(\Omega)}).
$$
\label{adn_global}
\end{theorem} 

We will also be using the following local analogue of this result. For a standard proof see \cite{ReichelWeth}.

\begin{corollary} 
Let $\Omega$ be a ball $\{x\in \R^N: |x|<R\}$ or a half-ball 
$\{x\in \R^N: |x|<R, x_1 >0\}$. Let $m\in \N$, 
$a_{ij}\in C^{2m-2}(\overline{\Omega})$, $b_\alpha \in L^\infty(\Omega)$, 
$g\in L^p(\Omega)$ for some $p\in (1,\infty)$. Suppose 
$u\in W^{2m,p}(\Omega)$ satisfies \eqref{linear_eq} 
\begin{itemize}
\item[(i)] either on the ball 
\item[(ii)] or on the half-ball together with the boundary conditions 
$u=\frac{\partial}{\partial x_1}u=\ldots=
\frac{\partial^{m-1}}{\partial x_1^{m-1}} u =0$ on 
$\{x\in\R^N: |x|<R, x_1=0\}$.
\end{itemize}
Then there exists a constant $C>0$ depending only on 
$\|a_{ij}\|_{C^{2m-2}}, \|b_\alpha\|_\infty, \lambda, \Omega, N, p, m$, the 
modulus of continuity of $a_{ij}$ and $R$ such that for any $\sigma\in (0,1)$ 
$$
\|u\|_{W^{2m,p}(\Omega\cap B_{\sigma R})} \leq \frac{C}{(1-\sigma)^{2m}}
(\|g\|_{L^p(\Omega)}+ \|u\|_{L^p(\Omega)}).
$$
\label{adn_local}
\end{corollary}

It is sometimes convenient to rewrite the operator $L$ in the form 
$$
L=(-1)^m \sum_{|\alpha|=2m} a_\alpha(x)D^\alpha + \sum_{0\leq |\alpha|\leq 2m-1} c_\alpha(x)D^\alpha.
$$
Here $a_\alpha(x)= \sum \limits_{I \in \cM_\alpha}a_{i_1 i_2}(x)\cdot 
a_{i_3 i_4}(x)\cdots a_{i_{2m-1} i_{2m}}(x)$, where $\cM_\alpha$ is the set of all vectors 
$I=(i_1,\dots,i_{2m}) \in \{1,\dots,N\}^{2m}$ satisfying $\#\{j\::\:i_j=l\}= \alpha_l$ for $l=1,\dots,N$. Note that $a_\alpha, c_\alpha$ are uniformly $\alpha$-H\"older continuous in $\Omega$. 

\medskip

Finally, the following lemma is used a number of times in the subsequent proof of Theorem~\ref{non_ex}. A version of part (a) of the lemma already appeared in Reichel, Weth \cite{ReichelWeth} and similar arguments have been used by Wei and Xu in \cite{WeiXu}.

\begin{lemma} (a) Let $v$ be a strong $W^{2m,1}_{loc}(\R^N)\cap C^{2m-1}(\R^N)$ solution of $(-\Delta)^m v \geq g(v)$ in $\R^N$ such that $D^\alpha v$ is bounded for all multi-indices $\alpha$ with $0\leq \alpha\leq 2m-1$. If $g:\R\to [0,\infty)$ is convex and non-negative with $g(s)>0$ for $s<0$ then either $v>0$ or $v\equiv 0$. \\
(b) Let $v$ be a strong $W^{2m,1}_{loc}(\R^N_+)\cap C^{2m-1}(\R^N_+)$ solution of $(-\Delta)^m v \geq 1$ in $\R^N_+$. Then $(-\Delta)^{m-1} v$ is unbounded.
\label{Liou}
\end{lemma}

\begin{proof} {\em Part (a):} Let $v_l := (-\Delta)^l v$ for $l=1,\ldots, m-1$ and set $v_0=v$. Then we have 
$$
-\Delta v_0 = v_1, \quad -\Delta v_1 = v_2, \quad \ldots \quad -\Delta v_{m-1}\geq g(v_0) \mbox{ in } 
\R^N. 
$$
First we show that $v_l \geq 0$ in $\R^N$ for $l=1,\ldots,m-1$. Assume that there exists $l\in \{1,\ldots,m-1\}$ and 
$x_0\in \R^N$ with $v_l(x_0)<0$ but $v_j\geq 0$ in $\R^N$ for $j=l+1,\ldots,m$. 
We may assume w.l.o.g. that $x_0=0$. If we define for a function $w\in W_{loc}^{2,1}(\R^N)\cap C^1(\R^N)$ spherical averages $\bar w(x) = \frac{1}{r^{N-1}\omega_N} \oint_{\partial B_r(0)} w(y) 
\,d\sigma_y$, $r=|x|$ then the radial functions $\bar v_0, \bar v_1,\ldots,\bar v_{m-1}$ satisfy 
$$
-\Delta \bar v_0 = \bar v_1, \quad -\Delta \bar v_1 = \bar v_2, \quad \ldots \quad 
-\Delta \bar v_{m-1} \geq  g(\bar v_0) \mbox{ in } \R^N,
$$
where we have used Jensen's inequality and the convexity of $g$. Since $v_l(0)<0$ we also have $\bar v_l(0)<0$. Moreover
\begin{align}
\bar v_l'(r) =& \frac{1}{\omega_N}\oint_{\partial B_1(0)} (\nabla v_l)(r\xi)\cdot\xi\,d\sigma_\xi\label{sys}\\
 =& \frac{1}{\omega_N}\int_{B_1(0)} (\Delta v_l)(r\xi)r\,d\xi
\;\left\{
\begin{array}{ll}
= \displaystyle\frac{-1}{\omega_N}\int_{B_1(0)} v_{l+1}(r\xi)r\,d\xi & \mbox{ if } l<m-1, \vspace{\jot}\\
\leq \displaystyle\frac{-1}{\omega_N}\int_{B_1(0)} g(v(r\xi))r\,d\xi & \mbox{ if } l=m-1.
\end{array}
\right. \nonumber
\end{align}
%
Since the right-hand side is non-positive in both cases we obtain $\bar v_l(r)\leq \bar v_l(0)<0$. Integrating the inequality
$$
\Delta \bar v_{l-1}= -\bar v_l\geq -\bar v_l(0)>0
$$
we obtain $r^{N-1} \bar v_{l-1}'(r)\geq -\frac{r^N}{N} \bar v_l(0)$, i.e, $\bar v_{l-1}'(r) \geq 
-\frac{r}{N} \bar v_l(0)$. The unboundedness of $\bar v_{l-1}'$ yields a contradiction. 

\smallskip

Next we show that $v=v_0\geq 0$. Assume that $v_0(x_0)<0$ and w.l.o.g. $x_0=0$. Since 
$\Delta \bar v_0=-\bar v_1\leq 0$ we see that $\bar v_0'(r)\leq 0$ and we define  
$\alpha:= \lim_{r\to\infty} \bar v_0(r)<0$. Thus $g(\bar v_0(r))\geq \frac{1}{2}g(\alpha)>0$ for $r\geq r_0$. As in \eqref{sys} we find
\begin{align*}
\bar v_{m-1}'(r) \leq & \frac{-1}{\omega_N}\int_{B_1(0)} g(v_0(r\xi))r\,d\xi = \frac{-1}{r^{N-1}\omega_N}\int_{B_r(0)} g(v_0(\eta))\,d\eta \\
= & \frac{-1}{r^{N-1}\omega_N}\int_0^r \oint_{B_s(0)} g(v_0(\eta))\,d\sigma_\eta\,ds \leq -\int_0^r \frac{s^{N-1}}{r^{N-1}} g(\bar v_0(s))\,ds \\
\leq & -\int_{r/2}^r \frac{s^{N-1}}{r^{N-1}} \frac{g(\alpha)}{2}\,ds 
\end{align*}
if $r\geq 2r_0$. Since the last term converges to $-\infty$ as $r\to \infty$ we obtain a contradiction to the boundedness of $\bar v_{m-1}'$. Finally the alternative $v>0$ or $v\equiv 0$ follows since $-\Delta v\geq 0$ by the first part of the proof.

\medskip

\noindent
{\em Part (b):}
Let $w := (-\Delta)^{m-1} v$ so that $w$ is a strong $W^{2,1}_{loc}(\R^N_+)\cap C^1(\R^N_+)$ solution of $-\Delta w \geq 1$. Let 
$$
\bar w (r;X) := \frac{1}{r^{N-1} \omega_N}\oint_{\partial B_r(X)} w(y)\,d\sigma_y = \frac{1}{\omega_N} \oint_{\partial B_1(0)} w(X+r\xi)\,d\sigma_\xi
$$
for $X\in \R^N_+$ and $0<r<X_1$. For fixed $X\in \R^N_+$ the function $\bar w$ satisfies 
$$
\bar w'(r) = \frac{1}{\omega_N}\oint_{\partial B_1(0)} (\nabla w)(X+r\xi)\cdot \xi\,d\sigma_\xi = \frac{1}{\omega_N} \int_{B_1(0)} (\Delta w)(X+r\xi)r\,d\xi \leq \frac{-r}{N}
$$
for $0<r<X_1$. Hence, $\bar w(r) \leq \bar w(0)- \frac{r^2}{2N}$. Letting $X_1$ and $r$ tend to infinity we find that $w$ cannot stay bounded.
\end{proof}

We now have all the tools to complete the\\[0.2cm]
\noindent
{\em Proof of Theorem \ref{non_ex}.} Assume for contradiction that there exists a sequence of pairs $(u_k,\lambda_k)$ of solutions of \eqref{basic_lambda} with $\lambda_k\to\infty$ for $k\to\infty$. Let $M_k := \|u_k\|_\infty$. By considering a suitable subsequence we can assume that there exists $x_k\in \Omega$ such that either $M_k=u_k(x_k)$ for all $k\in\N$ or 
$M_k = -u_k(x_k)$ for all $k\in \N$. 

\smallskip

\noindent
\underline{Case 1:} $\|u_k\|_\infty$ stays bounded. W.l.o.g. we can assume $0\in\Omega$ and $B_\delta(0)\subset\Omega$ for some $\delta>0$. Set $v_k(x) := u_k(\lambda_k^{-1/2m} x)$. Then $v_k$ satisfies
$$
\bar L^k v_k(x) = \frac{1}{\lambda_k} f(\lambda_k^{-1/2m}x,v_k)+1 \mbox{ in } B_{\lambda_k^{1/2m}\delta}(0)
$$
where 
$$
\bar L^k :=(-1)^m \sum_{|\alpha|=2m} \bar a_\alpha^k(\lambda_k^{-1/2m}x) D^\alpha
+\sum_{0\leq |\alpha|\leq 2m-1} \lambda_k^{\frac{|\alpha|}{2m}-1}\bar c_\alpha^k(\lambda_k^{-1/2m}x) D^\alpha.
$$
By standard interior regularity on the ball $B_R(0)$ for any $R>0$ and any $p\geq 1$ there exists a constant $C_{p,R}>0$ such that 
$$
\|v_k\|_{W^{2m,p}(B_R(0))} \leq C_{p,R} \mbox{ uniformly in } k.
$$
For $p$ sufficiently large and by passing to a subsequence (again denoted $v_k$) we see that 
$v_k \to v$ in $C^{2m-1,\alpha}_{loc}(\R^N)$ and in $W^{m,p}_{loc}(\R^N)$ as $k\to \infty$ for every 
$R>0$, where $v\in C^{2m-1,\alpha}_{loc}(\R^N)\cap W^{m,p}_{loc}(\R^N)$ is a bounded weak (and hence classical) solution of 
$$
{\mathcal L} v= 1 \mbox{ in } \R^N, \quad \mbox{ where} \quad 
{\mathcal L} = (-1)^m \sum_{|\alpha|=2m} a_\alpha(0) D^\alpha
= \Bigl(-\sum_{i,j=1}^N a_{ij}(0)\frac{\partial^2}{\partial x_i\partial x_j} \Bigr)^m. 
$$
By a linear change of variables we may assume that $v$ is a bounded, classical, entire solution of 
$(-\Delta)^m v = 1$ in $\R^N$. Lemma \ref{Liou}(b) shows that we have reached a contradiction.

\noindent
\underline{Case 2:} $\|u_k\|_\infty$ is unbounded. For this case we
need to discuss various sub-cases depending on the growth of the numbers
$$
\rho_k:=M_k^\frac{q-1}{2m}\dist(x_k,\partial\Omega), \qquad k \in \N.
$$ 
Passing to a subsequence, we may assume that either $\rho_k \to
\infty$ or $\rho_k \to \rho \ge 0$ as $k \to \infty$. 
\noindent \underline{Case 2.1:} $\rho_k \to\infty$ as $k\to \infty$. Again we need to distinguish two further possibilities. Let $\tilde\lambda_k := \lambda_k/M_k^q$.

\smallskip

\noindent
\underline{Case 2.1.a:} $\tilde\lambda_k$ is bounded, i.e., up to selecting a subsequence, $\tilde\lambda_k\to \lambda^\ast\geq 0$. Then we set $v_k(y) := \frac{1}{M_k}u_k(M_k^\frac{1-q}{2m}y+x_k)$ so that $\|v_k\|_\infty=1$ and either $v_k(0)=1$ for all $k\in\N$ (positive blow-up) or $v_k(0)=-1$ for all $k\in\N$ (negative blow-up). Moreover we can assume that $x_k\to \bar x \in \overline{\Omega}$. The functions $v_k$ are well-defined on the sequence of balls $B_{\rho_k}(0)$ as $k\to \infty$ and they satisfy 
$$
\bar L^k v_k(y) = \frac{1}{M_k^q} \Bigl( \underbrace{f(M_k^\frac{1-q}{2m}y+x_k,M_kv_k(y))}_{=:f_k(y)}+\lambda_k\Bigr)\qquad \text{for $y\in B_{\rho_k}(0)$,}
$$
where this time
$$
\bar L^k :=(-1)^m \sum_{|\alpha|=2m} \bar a_\alpha^k(y) D^\alpha
+\sum_{0\leq |\alpha|\leq 2m-1} \bar c_\alpha^k(y) D^\alpha
$$
and 
$$
\bar a_\alpha^k(y) := a_\alpha(M_k^\frac{1-q}{2m} y+x_k), \quad 
\bar c_\alpha^k(y) := M_k^{(q-1)(\frac{|\alpha|}{2m}-1)} 
c_\alpha(M_k^\frac{1-q}{2m} y+x_k).
$$
By our assumption (H1) on the nonlinearity $f(x,s)$ we have that $\|f_k\|_{L^\infty(B_{\rho_k}(0))}$ 
is bounded in $k$. Note that the ellipticity constant, the $L^\infty$-norm of the coefficients of $\bar L^k$ and the moduli of continuity of $\bar a_\alpha^k$ are not larger then the one for the operator $L$. By applying Corollary~\ref{adn_local} on the ball $B_R(0)$ for any $R>0$ and any $p\geq 1$ there exists a constant $C_{p,R}>0$ 
such that 
$$
\|v_k\|_{W^{2m,p}(B_R(0))} \leq C_{p,R} \mbox{ uniformly in } k.
$$
For large enough $p$ we may extract a subsequence (again denoted $v_k$) such 
that $v_k \to v$ in $C^{2m-1,\alpha}(B_R(0))$ as $k\to \infty$ for every 
$R>0$, where $v\in C^{2m-1,\alpha}_{loc}(\R^N)$ is bounded with 
$\|v\|_\infty=1=\pm v(0)$. Taking yet another subsequence we may assume that 
$f_k\overset{\ast}{\rightharpoonup} F$ in $L^\infty(K)$ as $k\to \infty$ for every compact set 
$K\subset \R^N$. Also we see that 
\begin{equation}
F(y) = \left\{\begin{array}{ll}
h(\bar x)v(y)^q & \mbox{ if } v(y)>0,\vspace{\jot}\\
k(\bar x)|v(y)|^q & \mbox{ if } v(y)<0,
\end{array} \right.
\label{def_F}
\end{equation}
because, e.g., if $v(y)>0$ then there exists $k_0$ such that $v_k(y)>0$ for $k\geq k_0$ and hence 
$M_kv_k(y)\to \infty$ as $k\to \infty$. Therefore (H1) implies that 
$f_k(y)\to h(\bar x) v(y)^q$ as $k\to \infty$, and a similar pointwise convergence holds at 
points where $v(y)<0$.  Finally, note that the pointwise convergence of 
$f_k$ on the set $Z^+=\{y\in \R^N: v(y)>0\}$ and $Z^-=\{y\in \R^N:v(y)<0\}$ determine due to 
the dominated convergence theorem the 
weak$\ast$-limit $F$ of $f_k$ on the set $Z^+\cup Z^-$. Since $\bar c_\alpha^k(y)\to 0$ and 
$\bar a^k_\alpha(y)\to a_\alpha(\bar x)$ as $k\to\infty$ and since, for any fixed $p\in (1,\infty)$, we may assume that 
$v_k\to v$ in $W^{2m-1,p}_{loc}(\R^N)$ we find that $v$ is a bounded, 
weak $W^{m,p}_{loc}(\R^N)$-solution of 
\begin{equation} 
{\mathcal L} v= F+\lambda^\ast \mbox{ in } \R^N, \quad \mbox{ where} \quad 
{\mathcal L} = (-1)^m \sum_{|\alpha|=2m} a_\alpha(\bar x) D^\alpha
= \Bigl(-\sum_{i,j=1}^N a_{ij}(\bar x)\frac{\partial^2}{\partial y_i\partial y_j} \Bigr)^m. 
\label{eq:whole_space}
\end{equation}
Since $F\in L^\infty(\R^N)$ we get that 
$v\in W^{2m,p}_{loc}(\R^N)\cap C^{2m-1,\alpha}_{loc}(\R^N)$ is a bounded, 
strong solution of \eqref{eq:whole_space}. Because $D^{2m} v = 0 $ a.e. on the set 
$\{y\in \R^N: v(y)=0\}$ one finds that $v$ is a strong solution of 
\begin{equation}
{\mathcal L} v=
\left\{\begin{array}{ll}
h(\bar x)v(y)^q+\lambda^\ast & \mbox{ if } v(y)>0,\vspace{\jot}\\
0 & \mbox{ if } v(y) =0, \vspace{\jot}\\
k(\bar x)|v(y)|^q+\lambda^\ast & \mbox{ if } v(y)<0
\end{array}\right.
\label{right_hand_side_whole_space}
\end{equation}
in $\R^N$. Note that the right-hand side of \eqref{right_hand_side_whole_space} is larger or equal to $g(v)$, where the function $g$ is defined by
\begin{equation}
g(s):= \left\{\begin{array}{ll}
h(\bar x)s^q & \mbox{ if } s\geq 0,\vspace{\jot}\\
k(\bar x)|s|^q & \mbox{ if } s \leq 0.
\end{array} \right.
\label{def_g}
\end{equation}
Since the function $g$ is convex we can apply Lemma~\ref{Liou}(a) and obtain $v>0$. Thus $v$ is a classical $C^{2m,\alpha}_{loc}(\R^N)$ solution, and by a linear change of variables we may assume that $v$ solves 
$$
(-\Delta)^m v = h(\bar x)v^q+\lambda^\ast \mbox{ in } \R^N, \qquad v(0)=1.
$$
Clearly $v$ and all its derivatives of order $\leq 2m$ are bounded. If $\lambda^\ast=0$ then Theorem~\ref{rn} tells us that this is impossible. And if $\lambda^\ast>0$ then Lemma ~\ref{Liou}(b) provides a contradiction. This finishes the proof in this case.

\smallskip

\noindent
\underline{Case 2.1.b:} $\tilde\lambda_k$ is unbounded, i.e., up to a subsequence $\tilde\lambda_k\to \infty$. Now we set $v_k(y) := \frac{1}{M_k}u_k(M_k^\frac{1-q}{2m}\tilde\lambda_k^{-1/2m}y+x_k)$. The functions $v_k$ are again well defined on a sequence of expanding balls and satisfy
\begin{equation}
  \label{eq:new-eq-2.2-b}
\bar L^k v_k(y) = \frac{1}{\tilde\lambda_kM_k^q} \Bigl( \underbrace{f(M_k^\frac{1-q}{2m}\tilde\lambda_k^{-1/2m}y+x_k,M_kv_k(y))}_{=:f_k(y)}
+\lambda_k\Bigr), 
\end{equation}
where this time
$$
\bar L^k :=(-1)^m \sum_{|\alpha|=2m} \bar a_\alpha^k(y) D^\alpha
+\sum_{0\leq |\alpha|\leq 2m-1} \bar c_\alpha^k(y) D^\alpha
$$
with
$$
\bar a_\alpha^k(y) := a_\alpha(M_k^\frac{1-q}{2m}\tilde\lambda_k^{-1/2m}y+x_k), \quad 
\bar c_\alpha^k(y) := M_k^{(q-1)(\frac{|\alpha|}{2m}-1)} \tilde\lambda_k^{\frac{|\alpha|}{2m}-1}
c_\alpha(M_k^\frac{1-q}{2m}\tilde\lambda_k^{-1/2m}y+x_k).
$$
Arguing like before we arrive at the situation that $v_k\to v$ in
$W^{2m-1,p}_{loc}(\R^N)$ as $k\to \infty$, where, modulo a linear
change of variables, $v\in C^{2m-1,\alpha}_{loc}(\R^N)\cap W^{2m,p}_{loc}(\R^N)$ is a bounded strong (and hence classical) solution of $(-\Delta)^m v = 1$ in $\R^N$. A contradiction is reached via Lemma~\ref{Liou}(b).

\smallskip

\noindent
\underline{Case 2.2:}
$\rho_k \to\rho\geq 0$. Then, modulo a subsequence, $x_k\to \bar x\in\partial\Omega$ as $k\to \infty$, and after translation we may assume that $\bar x= 0$. By flattening the boundary through a local change of coordinates we may assume that near $\bar x=0$ the boundary is contained in the hyperplane $x_1=0$, and that $x_1>0$ corresponds to points inside $\Omega$. Since $\partial\Omega$ is locally a $C^{2m,\alpha}$-manifold, this change of coordinates transforms the operator $L$ into a similar operator which satisfies the same hypotheses as $L$. For simplicity we 
call the transformed variables $x$, the transformed domain $\Omega$
and the transformed operator $L$. Note that
$\dist(x_k,\partial\Omega)=x_{k,1}$ for sufficiently large $k$. By
passing to a subsequence we
may assume that this is true for every $k$, so that $\rho_k=M_k^\frac{q-1}{2m}x_{k,1}$.  As before we need to distinguish two further possibilities by defining $\tilde\lambda_k := \lambda_k/M_k^q$.

\smallskip

\noindent
\underline{Case 2.2.a:} Up to selecting a subsequence assume that
$\tilde\lambda_k\to \lambda^\ast\geq 0$. In this case we define the
function $v_k$, the coefficients $\bar a_\alpha^k, \bar c_\alpha^k$
and the operator $\bar L^k$ as in Case~2.1.a, where $v_k$ is now
defined on the set $\{y \in \R^N\::\:M_k^\frac{1-q}{2m}y+x_k \in
\Omega\}$ which contains $B_{\rho_k}(0)$. Then we make another change of coordinates, defining 
\begin{align*}
w_k(z) &:= v_k(z-\rho_k e_1),\\ 
\tilde a_\alpha^k(z) &:= \bar a_\alpha^k(z-\rho_k e_1),\\
\tilde c_\alpha^k(z) &:= \bar c_\alpha^k(z-\rho_k e_1),
\end{align*}
where $e_1=(1,0,\dots,0) \in \R^N$ is the first coordinate vector, and likewise the
operator $\tilde L^k$. Note that $w_k$, $\tilde a_\alpha^k$, $\tilde c_\alpha^k$ 
and the operator $\tilde L^k$ are defined on the set 
$$
\Omega_k:=\{z \in
\R^N\::\:M_k^\frac{1-q}{2m}z +(0,x_{k,2},\dots,x_{k,N}) \in \Omega\},
$$
and that $w_k(\rho_k e_1)=\pm 1$. We now fix $R>0$ and let $B_R^+ = B_R(0)\cap \R^N_+$. By our assumptions on the boundary $\partial \Omega$ near $\bar x$,
we have $B_R^+ \subset \Omega_k$ for sufficiently large $k$. Moreover,
$w_k$ satisfies 
$$
\tilde L^k w_k(z)=\tilde f_k(z)+\tilde\lambda_k \mbox{ in } B_R^+, \quad \mbox{ where } 
\tilde f_k(z):= \frac{1}{M_k^q} f(M_k^\frac{1-q}{2m}z+(0,x_{k,2},\ldots,x_{k,n}),M_k w_k(z)),
$$
together with Dirichlet-boundary conditions on $\{z\in \R^N:|z|<R, z_1=0\}$. Hence we may 
apply Corollary~\ref{adn_local} on the half-ball 
$B_R^+$ and find that for any $p\geq 1$ there exists a constant $C_{p,R}>0$ 
such that 
$$
\|w_k\|_{W^{2m,p}(B_R^+)} \leq C_{p,R} \mbox{ uniformly in } k.
$$
By the Sobolev embedding theorem, this implies that $\nabla v_k$ is
bounded on $B_R^+$ independently of $k$, and since 
$$
1 = |\underbrace{v_k(0)}_{=\pm 1}-\underbrace{v_k(\rho_k,0,\ldots,0)}_{=0}| 
\leq \rho_k \|\nabla v_k\|_\infty,
$$
we see that $\rho=\lim \limits_{k \to \infty}\rho_k>0$. As in Case~2.1.a we can now extract convergent subsequences $w_k\to w$ in 
$C^{2m-1,\alpha}_{loc}(\overline{\R^N_+})$ and 
$f_k\overset{\ast}{\rightharpoonup} F$ in $L^\infty(\R^N_+)$ as $k\to \infty$, 
where $F\geq 0, \not \equiv 0$ is determined in the same way as in Case~2.1.a. This time, $w$ is a bounded, strong $W^{2m,p}_{loc}(\R^N_+)\cap 
C^{2m-1,\alpha}_{loc}(\overline{\R^N_+})$-solution of 
$$
{\mathcal L} w = F+\lambda^\ast \mbox{ in } \R^N_+, \qquad 
\frac{\partial}{\partial z_1}w=\ldots=
\frac{\partial^{m-1}}{\partial z_1^{m-1}} w =0 \mbox{ on }\partial \R^N_+
$$
with ${\mathcal L}$ as in \eqref{eq:whole_space}. 
By a linear change of variables we may assume that $w$ solves 
\begin{equation}
(-\Delta)^m w = g(w)+\lambda^\ast \mbox{ in } \R^N_+, \qquad \frac{\partial}{\partial z_1}w=\ldots=
\frac{\partial^{m-1}}{\partial z_1^{m-1}} w =0 \mbox{ on }\partial \R^N_+,
\label{eq_halfspace}
\end{equation}
where $g$ is defined as in \eqref{def_g} of Case~2.1.a. The
representation formula of Theorem 9 
in \cite{ReichelWeth} applies and shows that $w$ is nonnegative, so
that $g(w(z))=h(\bar x)w(z)^q$. Moreover,  
$$
w(\rho e_1)=\lim_{k \to \infty} w_k(\rho_k e_1)=1,
$$
so that $w$ is a positive, bounded and classical solution
$C^{2m}$-solution of $(-\Delta)^m w= h(\bar x) w^q+\lambda^\ast$ in
$\R^N_+$ with Dirichlet boundary conditions on $\partial \R^N_+$.  A contradiction is reached by either Theorem~\ref{rn_plus} if $\lambda^\ast=0$ or Lemma~\ref{Liou}(b) if $\lambda^\ast>0$.

\smallskip

\noindent
\underline{Case 2.2.b:} Up to selecting a subsequence $\tilde\lambda_k\to\infty$.
In this case we need to define $v_k$ as in Case~2.1.b, which is now well-defined on the set 
$$
\Sigma_k:= \{y \in \R^N\::\:M_k^\frac{1-q}{2m}\tilde\lambda_k^{-1/2m}y+x_k\in
\Omega\}
$$ 
and satisfies \eqref{eq:new-eq-2.2-b} on this set. By our assumptions
on the boundary of $\Omega$ near $\bar x$, we have 
$$
\dist(0,\partial \Sigma_k)=M_k^\frac{q-1}{2m} \tilde\lambda_k^{1/2m} x_{k,1}=
\rho_k \tilde\lambda_k^{1/2m}=:\tau_k \qquad \text{for all $k$.}
$$
Passing to a subsequence, we may assume that either $\tau_k \to \infty$ or
$\tau_k \to \tau\ge 0$ as $k \to \infty$. In the former case we come to a
contradiction as in Case 2.1.b, since then $v_k$ is well defined and
bounded on a sequence of expanding balls. In the latter case we
proceed completely analogously as in Case 2.2.a with $\rho_k$
replaced by $\tau_k$ for every $k$. The only difference is that in
this case, modulo a linear
change of variables, we end up with a bounded strong classical
solution of $(-\Delta)^m v = 1$ in $\R^N_+$. Again a contradiction is
reached via Lemma~\ref{Liou}(b).

Since in all cases we obtained a contradiction, the proof of
Theorem~\ref{non_ex} is finished.\qed

\section{Proof of the existence result} \label{sec:ex}

In this section we complete the proof of Theorem \ref{non_ex}. Finding a solution $u\in C^{2m,\alpha}(\overline{\Omega})$ of \eqref{basic_lambda} is equivalent to finding a solution $u\in C^\alpha(\overline{\Omega})$ of the equation
\begin{equation}
[\Id-\cK_\lambda](u)= 0
\label{equiv}
\end{equation}
where for $\lambda \in \R$ the nonlinear operator $\cK_\lambda:C^\alpha(\overline{\Omega})\to 
C^\alpha(\overline{\Omega})$ is defined by
$$
\cK_\lambda(u)= L^{-1}w \quad \text{with}\quad w(x)=f(x,u(x))+\lambda.
$$
By assumption (H3) we may regard $L^{-1}: C^\alpha(\overline{\Omega})\to 
C^{2m,\alpha}(\overline{\Omega})$ as a bounded linear operator. Moreover, 
since the embedding $C^{2m,\alpha}(\overline{\Omega}) \hookrightarrow C^\alpha(\overline{\Omega})$ is compact, $\cK_\lambda$ is also compact for every $\lambda \in \R$. Let $\Lambda>0$ be as in 
Theorem \ref{non_ex} so that \eqref{equiv} has no solution for 
$\lambda\geq \Lambda$. By Theorem~\ref{reichel-weth-alt} there exists  
$K>0$ such that for all $\lambda\in [0,\Lambda]$ any solution $u\in C^{\alpha}(\overline{\Omega})$ of \eqref{equiv} satisfies $\|u\|_\infty\leq K$. By elliptic regularity and (H3) we may assume $\|u\|_{C^{\alpha}(\overline{\Omega})}\leq K$ by adjusting $K$. Consequently, we find that 
\begin{equation}
\label{injectivity}
[\Id-\cK_\lambda](u) \not= 0 \qquad \text{if $(u,\lambda) \in \left(B_{2K}(0) \times \{\lambda\}\right)\, \cup \,\left(
\partial B_{2K}(0) \times [0,\Lambda]\right)$},
\end{equation}
where $B_{2K}(0) \subset C^\alpha(\overline{\Omega})$ denotes the 
$2K$-ball with respect to $\|\cdot\|_{C^{\alpha}(\overline{\Omega})}$. 
The homotopy invariance of the Leray-Schauder degree and and \eqref{injectivity} imply 
$$
\degree(\Id-\cK_0, B_{2K}(0),0) = \degree(\Id-\cK_\Lambda), B_{2K}(0),0)=0.
$$
For these and other properties of the Leray-Schauder degree, we refer the reader to \cite[Chapter 2.8]{deimling} or \cite[Chapter 2]{nirenberg}. Next we note that $0$ is an isolated solution of \eqref{equiv} for $\lambda=0$. Indeed, assume that there exists a sequence of solutions $u_n$ of \eqref{equiv} with $\lambda=0$ and $\|u_n\|_{C^{\alpha}}\to 0$ as $n\to \infty$. Let $v_n := u_n/\|u_n\|_\infty$. Since by (H2) $f(x,s)=O(s^2)$ uniformly in $x\in \Omega$ for $s$ in bounded intervals, we conclude that $L v_n = f(x,u_n)/\|u_n\|_\infty \to 0$ as $n\to \infty$ so that $\|v_n\|_\infty\to 0$ as $n\to \infty$, which is a contradiction. Moreover, since $\partial_s f(x,0)=0$ by (H2), the derivative $d \cK_0(0): C^\alpha(\overline{\Omega})\to 
C^\alpha(\overline{\Omega})$ of $\cK_0$ at $u=0$ vanishes. Hence, for small $\epsilon>0$, we have by \cite[Theorem 2.8.1]{nirenberg}
$$
\degree(\Id-\cK_0, B_\epsilon(0),0)= \degree(\Id- d \cK_0(0), B_\epsilon(0),0) = \degree(\Id, B_\epsilon(0),0)=1.
$$
The additivity property of the topological degree now implies that 
$$
\degree(\Id-\cK_0,B_{2K}(0)\setminus \overline{B_\epsilon(0)},0) \not= 0,
$$ 
hence there exists $u \in B_{2K}(0)\setminus \overline{B_\epsilon(0)}$ such that $u-\cK_0(u)= 0$. Therefore $u$ is a 
nontrivial solution of \eqref{equiv}.\qed

\section{Proof of Theorem 5} 
\label{no-boundedness}
In section we show how Theorem~\ref{rn_plus} can be deduced from
\cite[Theorem 4]{ReichelWeth} with the help of the doubling lemma of
Pol\'{a}\v{c}ik,  Quittner and Souplet \cite{PQS}. We recall the
following simple special case of this useful lemma. 

\begin{lemma} (cf. \cite{PQS})
\label{doubling-simplified}
Let $(X,d)$ be a complete metric space and $M: X \to (0,\infty)$ be
bounded on compact subsets of $X$. Then for any $y \in X$ and any
$k>0$ there exists $x \in X$ such that 
$$
M(x) \ge M(y) \qquad \text{and}\qquad M(z) \le 2 M(x)\quad \text{for all $z  \in B_{k/M(x)}(x)$}.
$$
\end{lemma}

This follows by taking $D= \Sigma= X$ in \cite[Lemma
5.1]{PQS}, so that $\Gamma:=\Sigma \setminus D= \emptyset$ and
therefore $\dist(y,\Gamma)=\infty$ for all $y \in X$.

\medskip

We now may complete the proof of Theorem~\ref{rn_plus}. Suppose by
contradiction that there exists an unbounded solution $u$ of
\eqref{eq:liouville-half-space}, and put $M:=u^{\frac{q-1}{2m}}:
\overline{\R^N_+} \to \R$. Then there exists a sequence $(y_k)_k\subset \R^N_+$ such 
that $M(y_k) \to \infty$ as $n \to \infty$. By
Lemma~\ref{doubling-simplified}, applied within the underlying
complete metric space $X:=\overline{\R^N_+}$, there exist another sequence $(x_k)_k \subset \R^N_+$ such that 
$$
M(x_k) \ge M(y_k) \qquad \text{and}\qquad M(z) \le 2 M(x_k)\quad \text{for
  all $z \in B_{k/M(x_k)}(x_k) \cap \overline{\R^N_+}$.}
$$
We then define $\rho_k:= x_{k,1}M(x_k)$, the affine halfspace 
$H_k:= \{\zeta \in \R^N\::\: \zeta_1 > -\rho_k\}$ and the function
$$
\tilde u_k: \overline{H}_k \to \R,\qquad \tilde u_k(\zeta)=\frac{u(x_k + \frac{\zeta}{M(x_k)})}{u(x_k)}   
$$
for $k \in \N$. Then $\tilde u_k$ is a nonnegative solution of  
\begin{equation}
\label{eq:begin-delta-tilde}
\left \{
  \begin{aligned}
   &(-\Delta)^m \tilde u_k=\tilde u_k^q &&\quad \text{in $H_k$,}\\
   &\tilde u_k= \frac{\partial}{\partial \zeta_1}\tilde u_k= \dots =
   \Bigl(\frac{\partial}{\partial \zeta_1}\Bigr)^{m-1}\tilde u_k= 0 &&\quad
   \mbox{ on }\partial H_k  
\end{aligned}
\right.
\end{equation}  
such that 
$$
\tilde u_k(0)=1 \quad \text{and}\quad \tilde u_k(\zeta) \le 2^{\frac{2m}{q-1}}
\quad \mbox{ for all } \zeta \in H_k \cap B_k(0).
$$ 
We may now pass to a subsequence and distinguish two cases:\\
\underline{Case 1:}  $\rho_k \to \infty$ as $k \to \infty$. In this
case Corollary~\ref{adn_local}(i) implies that
the sequence $(\tilde u_k)_k$ is locally $W^{2m,p}$-bounded on $\R^N$, therefore
we can extract a convergent subsequence $\tilde u_k \to \tilde u$ in 
$C^{2m-1,\alpha}_{loc}(\R^N)$, where $\tilde u$ is a
solution of \eqref{eq-rn} satisfying $u(0)=1$. By Theorem~\ref{rn} we obtain a contradiction.\\
\underline{Case 2:} $\rho_k \to \rho \ge 0$ as $k \to \infty$. In the case
we perform a further change of coordinates, defining 
$$
v_k(z):= \tilde u_k(z-\rho_k e_1) \mbox{ for } z \in
  \overline{\R^N_+}, 
$$
where again $e_1=(1,0,\dots,0) \in \R^N$ is the first coordinate vector. Then $v_k$ is a nonnegative solution of  
$$\left \{
  \begin{aligned}
   &(-\Delta)^m v_k=v_k^q &&\quad \text{in $\R^N_+$,}\\
   &v_k= \frac{\partial}{\partial z_1}v_k= \dots =
   \Bigl(\frac{\partial}{\partial z_1}\Bigr)^{m-1}v_k= 0 &&\mbox{ on }\partial \R^N_+,  
\end{aligned}
\right.
$$
while
$$
v_k(\rho_k e_1)=1 \quad \text{and}\quad v_k(z) \le 2^{\frac{2m}{q-1}}
\quad \text{for all $z \in B_k(\rho_k e_1) \cap \overline{\R^N_+}$.}
$$
Using now Corollary~\ref{adn_local}(ii), we deduce that
the sequence $(v_k)_k$ is locally $W^{2m,p}$-bounded in $\overline{\R^N_+}$. In particular
$|\nabla v_k|$ remains bounded independently of $k$ in a
neighborhood of the origin, which in view of the boundary conditions
implies that $\rho= \lim \limits_{k \to \infty}\rho_k>0$. 
We can therefore extract a convergent subsequence $v_k \to v$ in 
$C^{2m-1,\alpha}_{loc}(\overline{\R^N_+})$, where $v$ is a
solution of \eqref{eq:liouville-half-space} satisfying 
$$
v(\rho e_1) = 1 \qquad \text{and}\qquad v(z) \le 2^{\frac{2m}{q-1}} 
\quad \text{for $z \in \R^N_+$.}
$$
This contradicts \cite[Theorem 4]{ReichelWeth}, and the proof is finished.\qed



\end{document}